\documentclass[11pt,reqno, oneside]{amsart}

\usepackage[
  a4paper,
  right=30mm,
  left=30mm,
  top=30mm,
  bottom=30mm
]{geometry}

\usepackage{amsmath,amsthm,amsfonts,amssymb,amscd,mathtools,bm}
\usepackage{mathrsfs}
\usepackage[T1]{fontenc}
\usepackage[shortlabels]{enumitem}
\usepackage{empheq}
\usepackage{placeins}
\usepackage{xcolor}
\usepackage[unicode=true]{hyperref}
\usepackage{float}

\theoremstyle{plain}
\newtheorem{theorem}{Theorem}[section]
\newtheorem{lemma}[theorem]{Lemma}
\newtheorem{corollary}[theorem]{Corollary}
\newtheorem{proposition}[theorem]{Proposition}

\theoremstyle{definition}
\newtheorem{definition}[theorem]{Definition}
\newtheorem{example}[theorem]{Example}

\theoremstyle{remark}
\newtheorem{remark}[theorem]{Remark}

\numberwithin{equation}{section}

\newcommand{\rd}{\mathrm{d}}

\newcommand{\N}{\mathbb{N}}
\newcommand{\R}{\mathbb{R}}

\newcommand{\Q}{\mathbb{Q}}
\newcommand{\1}{\mathbf{1}}

\newcommand{\normA}[1]{\left\|#1\right\|_{\mathrm{A}}}
\newcommand{\normi}[1]{\left\|#1\right\|_{\infty}}
\newcommand{\Linf}{L^{\infty}}

\newcommand{\Lipz}{\mathrm{Lip}_0}        
\newcommand{\Czo}[1]{C_0([0,#1])}

\title[]{\boldmath Isometries and geometric liftings for Alexiewicz-normed $L^\infty$ spaces}

\author[]{Nuno J. Alves}
\address[N. J. Alves]{King Abdullah University of Science and Technology, CEMSE Division, Thuwal 23955-6900, Saudi Arabia.}
\email{nuno.januarioalves@kaust.edu.sa}

\begin{document}

\begin{abstract}
We study spaces of essentially bounded functions on compact subsets of the real line, equipped with the Alexiewicz norm given by the supremum norm of the primitive. Using the associated measure projection, we classify their surjective linear isometries as weighted composition operators determined by a sign and an increasing bi-Lipschitz map between the corresponding measure intervals. We also give geometric criteria characterizing when this interval-level map lifts to a homeomorphism or to a bi-Lipschitz homeomorphism between the underlying compact sets.
\end{abstract}

\maketitle
\thispagestyle{empty}

\section{Introduction}

The description of surjective linear isometries between function spaces is a classical theme in functional analysis. A paradigmatic result is the Banach--Lamperti theorem: if $1\le p<\infty$ and $p\neq 2$, every surjective linear isometry between $L^p$ spaces is a weighted composition operator. In Lamperti's formulation, the underlying change of variables is encoded at the level of measure algebras, together with a weight determined by a Radon--Nikod\'ym derivative; see \cite{lamperti1958isometries,fleming2002isometries}. Thus the theorem is essentially measure-theoretic and does not impose continuity-type regularity on the underlying transformation. Isometries have also been studied in nonlinear metric settings; see, for instance, \cite{geher2020isometric,virosztek_maps_survey}.

A different picture arises for function spaces associated with nonabsolutely convergent integration theories, such as the Denjoy and Kurzweil--Henstock integrals; see \cite{gordon1994integrals,henstock1991general}. In that setting, the natural quantity to measure is often not $\int |f|$, but the size of the indefinite integral. For an integrable function $f$ on an interval one considers its primitive $F(x)=\int_a^x f$ and measures $f$ by the supremum norm of $F$. This viewpoint goes back to Alexiewicz~\cite{alexiewicz1948linear}; see also~\cite{talvila2006continuity}. In recent work, De Pauw~\cite{de2025banach} obtained a Banach--Lamperti type theorem for the Kurzweil--Henstock integral, where the change of variables is forced to be bi-absolutely continuous.

In this paper we address an analogous question for Alexiewicz-normed spaces of essentially bounded functions on compact subsets of the real line. Let $K\subseteq\R$ be compact with positive Lebesgue measure $|K|:=\lambda(K)>0$. For $f\in L^\infty(K)$ we consider the primitive
\begin{equation} \label{primitive}
J_K f(x)\coloneqq\int_{[\min K,x]\cap K} f \, \rd \lambda,\qquad x\in K,
\end{equation}
and equip $L^\infty(K)$ with the Alexiewicz norm
\begin{equation}\label{norm_A}
\normA{f}\coloneqq\sup_{x\in K}\big|J_K f(x)\big|.
\end{equation}
This construction is available for arbitrary compact sets of positive measure, including sets for which the Kurzweil--Henstock formalism is not naturally available. 

Our approach is based on the measure projection
\begin{equation} \label{pi_K}
  \pi_K(x) \coloneqq \lambda\big([\min K,x]\cap K\big)\in[0,|K|],
  \qquad x\in K,
\end{equation}
which is Lipschitz, nondecreasing and surjective, and is constant on the gaps of $K$. An open interval $(\ell,r)\subseteq[\min K,\max K]$ is a \emph{gap of $K$} if $\ell,r\in K$ and $(\ell,r)\cap K=\emptyset$. Equivalently, gaps are the connected components of $[\min K,\max K]\setminus K$.

For an interval $[a,b]$, let $C_0([a,b])$ denote the space of real-valued continuous functions on $[a,b]$ vanishing at $a$, and let $\Lipz([a,b])$ denote its subspace of Lipschitz functions, both equipped with the supremum norm. Choosing the natural left-endpoint selector $\sigma_K:[0,|K|]\to K$ for the fibers of $\pi_K$,
\begin{equation}
\sigma_K(t)\coloneqq\min \pi_K^{-1}(t) = \min \big\{x \in K :\, \pi_K(x) = t \big\},
\end{equation}
we show that the primitive map $J_K$ identifies the Alexiewicz space with a Lipschitz space on the measure interval. More precisely, there is a linear isometric embedding
\begin{equation} \label{eq:intro_A_embed}
  (L^\infty(K),\normA{\cdot})\hookrightarrow C_0([0,|K|])
\end{equation}
whose range is exactly $\Lipz([0,|K|])$; see Proposition~\ref{prop:embed}. In particular, the completion of $(L^\infty(K),\normA{\cdot})$ is canonically isometric to $C_0([0,|K|])$.

This yields a Banach--Stone-type description of surjective isometries in terms of the associated measure intervals. The present results thereby complement the Banach--Lamperti-type theorem of De Pauw~\cite{de2025banach} in a related Alexiewicz-normed setting on compact subsets of~$\R$. Recall that a \emph{lipeomorphism} is a bi-Lipschitz homeomorphism. Our first result is the following representation theorem.

\begin{theorem}\label{thm:ABS}
Let $M,K \subseteq \R$ be compact sets of positive measure. A linear map
\[
T:(\Linf(K),\normA{\cdot}) \to (\Linf(M),\normA{\cdot})
\]
is a surjective isometry if and only if there exist $\varepsilon\in\{\pm1\}$ and an increasing lipeomorphism 
\begin{equation} \label{psi}
\psi:\ [0,|M|] \to  [0,|K|]
\end{equation}
such that the map $\phi: M \to K$ given by
\begin{equation} \label{phi}
\phi =\sigma_K \circ \psi \circ \pi_M,
\end{equation}
satisfies, for every $f\in \Linf(K)$,
\begin{equation}\label{thm_J_represent}
J_M(Tf)=\varepsilon\,(J_K f)\circ \phi.
\end{equation}
Equivalently, for a.e.\ $y\in M$,
\begin{equation}\label{thm_T_represent}
(Tf)(y)=\varepsilon\,(f\circ \phi)(y)\,(\psi'\circ \pi_M)(y).
\end{equation}
\end{theorem}

Theorem~\ref{thm:ABS} shows that the isometric classification is naturally reduced to the measure intervals through the projections $\pi_K$ and $\pi_M$. The geometric information of the compact sets reappears when one asks whether the map $\phi=\sigma_K\circ\psi\circ\pi_M$ can be chosen continuous or bi-Lipschitz. In general this need not happen, because the selector $\sigma_K$ may jump at fibers corresponding to endpoints of gaps.

To address this, we introduce two compatibility notions---fiber $\psi$-compatibility and gap $\psi$-compatibility---which, together with the representation theorem above, constitute the main structural contributions of the paper. These notions compare the fibers and gaps of the two measure projections. Roughly speaking, \emph{fiber $\psi$-compatibility} requires the nontrivial fibers of $\pi_M$ and $\pi_K$ to match, with the same order structure, under $\psi$, while \emph{gap $\psi$-compatibility} adds uniform comparability of the lengths of corresponding gaps. We show that these conditions provide exact geometric criteria for when there exists an increasing homeomorphism $\phi:M\to K$ satisfying $\pi_K\circ\phi=\psi\circ\pi_M$, and for when such a lift may in fact be chosen bi-Lipschitz.

The paper is organized as follows. Section~\ref{sec:preliminary} collects the necessary facts about measure projections and the Alexiewicz embedding into $C_0([0,|K|])$. Section~\ref{sec:proof_main} proves the representation theorem for surjective linear isometries. Sections~\ref{sec:fiber_compat} and~\ref{sec:gap_compat} develop fiber and gap compatibility and establish the homeomorphic and lipeomorphic lifting results.

\section{Preliminary material} \label{sec:preliminary}

This section collects the properties of the measure projection~$\pi_K$ and the
selector~$\sigma_K$ that will be used later, and recalls the identification of
$(L^\infty(K),\normA{\cdot})$ with $\Lipz([0,|K|])$.
\par 
Throughout this section, $K\subseteq\R$ is a compact set of positive measure. We write $\1_K$ for the indicator function of $K$.

\subsection{Measure projections and selector maps}  We record the elementary properties of $\Pi_K$, $\pi_K$, and $\sigma_K$ for later reference. Their proofs are straightforward and are omitted.

Let $\Pi_K$ denote the extension of $\pi_K$ to the interval $[\min K, \max K]$.
\begin{lemma} \label{lem:Pi_properties}
The extended projection $\Pi_K$ satisfies the following properties:
\begin{enumerate}[(i)]
\item $\Pi_K$ is Lipschitz continuous and \[\Pi_K'(x)=\1_K(x)\] for a.e.\ $x\in[\min K,\max K]$.
\item $\Pi_K$ is nondecreasing.
\item $\Pi_K$ is constant on each gap of $K$.
\item $\Pi_K$ is surjective.
\end{enumerate}
Consequently, the projection $\pi_K$ is Lipschitz continuous, nondecreasing and surjective.
\end{lemma}

\begin{lemma}\label{lem:left-selector}
The map $\sigma_K$ is well-defined, strictly increasing, left-continuous, and satisfies $\pi_K(\sigma_K(t))=t$ for all $t \in [0, |K|]$.
\end{lemma}

Now, let $E_K$ be the set 
\begin{equation} \label{set_E_K}
E_K\coloneqq \big\{t\in[0,|K|] : \#\pi_K^{-1}(t)>1 \big\}.
\end{equation}

\begin{lemma}\label{lem:properties_E_K}
The set $E_K$ has the following properties:
\begin{enumerate}[(i)]
\item $t\in E_K$ if and only if $t=\pi_K(\ell)=\pi_K(r)$ for some gap $(\ell,r)$ of $K$.
\item $E_K$ is countable; hence $|E_K|=0$.
\item $E_K$ is exactly the set of discontinuity points of the selector $\sigma_K$.
\end{enumerate}
\end{lemma}

\subsection{Alexiewicz isometric isomorphism}
In this subsection we establish the isometric embedding~\eqref{eq:intro_A_embed}. We define an auxiliary space $X_K$ of continuous functions on~$K$ by 
\begin{equation} \label{space_X_K}
X_K\coloneqq\Big\{F\in C(K) :\, F(\min K)=0, \ F|_{\pi_K^{-1}(t)} \text{ is constant for }  t\in[0,|K|]\Big\}.
\end{equation}
\par 

For $f\in L^\infty(K)$, the primitive $J_K f$ belongs to $X_K$: it is Lipschitz on $K$, vanishes at $\min K$, and is constant on each fiber of $\pi_K$.

We consider the maps $\Phi_K:\Czo{|K|}\to X_K$ and $\Psi_K:X_K\to \Czo{|K|}$ given by
\begin{equation} \label{map_Phi}
\Phi_K G(x)\coloneqq G(\pi_K(x)), \quad x \in K,
\end{equation}
and
\begin{equation} \label{map_Psi}
\Psi_K F(t) \coloneqq F(\sigma_K(t)), \quad t \in [0,|K|].
\end{equation}
Then $\Phi_K$ and $\Psi_K$ are inverse linear isometries. In particular, $X_K$ is canonically isometric to $C_0([0,|K|])$ and
\[
X_K=\{G\circ \pi_K:\, G\in C_0([0,|K|])\}.
\]

\begin{proposition} \label{prop:embed}
The composition
\begin{equation} \label{isometric_embedding}
(\Linf(K),\normA{\cdot})\xrightarrow{J_K} (X_K, \normi{\cdot})
\xrightarrow{\Psi_K} (\Czo{|K|},\normi{\cdot})
\end{equation}
is a linear isometric embedding whose range is $\Lipz([0,|K|])$. \par
Furthermore, for every $f \in L^\infty(K)$, 
\begin{equation} \label{deriv_Psi_J}
\frac{\rd }{\rd t} \Psi_K(J_K(f))(t) = f(\sigma_K(t)) \quad \text{for a.e. } t \in [0,|K|].
\end{equation}
 
\end{proposition}

\begin{proof}
The maps $J_K$ and $\Psi_K$ are linear, and so is their composition. Moreover, given $f \in L^\infty(K)$,
\begin{align*}
\sup_{t \in [0,|K|]} |\Psi_K(J_K f)(t)| & = \sup_{t \in [0,|K|]} |J_K f(\sigma_K(t))| = \sup_{x \in K} |J_K f (x)| = \| f\|_A
\end{align*}
and hence $\Psi_K \circ J_K$ is an isometry. \par 
Next, we prove that 
\[
\Psi_K\big(J_K(\Linf(K))\big)=\Lipz([0,|K|]).
\]
First, fix $f\in L^\infty(K)$ and note that since $\sigma_K(0)=\min K$, we have $\Psi_K(J_K f)(0) = J_K f(\min K)=0$. Moreover, for $0\leq s<t\leq |K|$, monotonicity of $\sigma_K$ yields $\sigma_K(s)\le \sigma_K(t)$; hence
\begin{align*}
|\Psi_K(J_K f)(t)-\Psi_K(J_K f)(s)| &=\Big|\int_{[\min K,\sigma_K(t)]\cap K} f - \int_{[\min K,\sigma_K(s)]\cap K} f\Big| \\
& = \Big|\int_{[\sigma_K(s),\sigma_K(t)]\cap K} f\Big|\\
&\le \|f\|_\infty\,\lambda\big([\sigma_K(s),\sigma_K(t)]\cap K\big).
\end{align*}
By the definition of $\pi_K$ and the fact that $\pi_K(\sigma_K(t))=t$,
\[
\lambda\big([\sigma_K(s),\sigma_K(t)]\cap K\big)
= \pi_K(\sigma_K(t))-\pi_K(\sigma_K(s))=t-s.
\]
Thus \[|\Psi_K(J_K f)(t)-\Psi_K(J_K f)(s)|\le \|f\|_\infty\,(t-s)\] for all $0\le s<t\le |K|$, so $\Psi_K(J_K f)\in \Lipz([0,|K|])$.
Therefore \[\Psi_K(J_K(\Linf(K)))\subseteq \Lipz([0,|K|]).\]\par
For the reverse inclusion, we let $G\in\Lipz([0,|K|])$ and note that there exists $g\in L^\infty([0,|K|])$ with
\[
G(t)=\int_0^t g(s) \, \rd s, \qquad 0 \leq t \leq |K|.
\]
Equivalently, $G^\prime = g$ a.e.\ in $[0,|K|]$. Define $f:=g  \circ  \pi_K\in L^\infty(K)$, and consider the following absolutely continuous functions on \([\min K,\max K]\):
\[
A(x)=\int_{\min K}^{x} (g \circ \Pi_K) \,\mathbf 1_K,\qquad
B(x)=G\big(\Pi_K(x)\big).
\]
Using that $\Pi_K'=\mathbf 1_K$ a.e., we get
\[
A'(x)=g(\Pi_K(x))\,\mathbf 1_K(x)=G'(\Pi_K(x))\,\Pi_K'(x)=B'(x)
\]
for a.e.\ $x \in [\min K, \max K]$, and $A(\min K)=B(\min K)=G(0)=0$. Hence $A=B$ on $[\min K,\max K]$. \par 
For any $t\in[0,|K|]$,
\begin{align*}
\Psi_K(J_K f)(t) & =J_K f(\sigma_K(t)) =\int_{[\min K,\sigma_K(t)]\cap K} g  \circ  \pi_K \\
& =A\big(\sigma_K(t)\big)=B\big(\sigma_K(t)\big) \\
& =G\big(\Pi_K(\sigma_K(t))\big)=G(t),
\end{align*}
because $\Pi_K=\pi_K$ on $K$, and $\pi_K(\sigma_K(t))=t$.
Therefore $G=\Psi_K(J_K f)$, establishing $\Lipz([0,|K|])\subseteq \Psi_K(J_K(\Linf(K)))$. Combining the two inclusions proves the claim. 

It remains to prove \eqref{deriv_Psi_J}. Let $H=\Psi_K(J_K f)\in\Lipz([0,|K|])$, and let $h\in L^\infty([0,|K|])$ denote its a.e.\ derivative, so that $H'=h$ a.e.\ on $[0,|K|]$.
By the surjectivity part just proved,
\[
H=\Psi_K(J_K(h\circ\pi_K)).
\]
Since $\Psi_K\circ J_K$ is injective, it follows that
\[
f=h\circ\pi_K \qquad \text{a.e.\ on } K.
\]
Let
\[
N=\{x\in K:\, f(x)\neq h(\pi_K(x))\}.
\]
Then $|N|=0$. Since $\pi_K$ is Lipschitz, also $|\pi_K(N)|=0$.

Let
\[
D=\{t\in[0,|K|]:\, H'(t)\ \text{exists and } H'(t)=h(t)\}.
\]
Then $|[0,|K|]\setminus D|=0$. Also, if $t\notin \pi_K(N)$, then necessarily
$\sigma_K(t)\notin N$, since otherwise
\[
t=\pi_K(\sigma_K(t))\in \pi_K(N),
\]
a contradiction. Therefore, for every
\[
t\in D\cap \big([0,|K|]\setminus \pi_K(N)\big),
\]
we have
\[
H'(t)=h(t)=h(\pi_K(\sigma_K(t)))=f(\sigma_K(t)).
\]
Since this intersection has full measure in $[0,|K|]$, \eqref{deriv_Psi_J} follows.
\end{proof}

\section{Representation of surjective isometries} \label{sec:proof_main}

Let $K,M\subseteq\R$ be compact with $|K|,|M|>0$.
Equip $\Linf(K),\Linf(M)$ with their Alexiewicz norms via $J_K,J_M$.
\par
We will use the following variant of the classical Banach--Stone theorem whose proof can be found in~\cite{de2025banach}.

\begin{lemma} \label{lem:banach_stone}
Let $a,b>0$ and $\widehat U:\Czo{a}\to \Czo{b}$ be a surjective linear isometry. Then there exist
$\varepsilon\in\{\pm1\}$ and a homeomorphism $\psi:[0,b]\to[0,a]$ with $\psi(0)=0$
such that
\[
\widehat U (G)=\varepsilon\,G \circ \psi
\]
for all $G \in \Czo{a}$.
\end{lemma}

We proceed with the proof of the theorem.

\begin{proof}[Proof of Theorem~\ref{thm:ABS}]
Assume that $T : \Linf(K) \to \Linf(M)$ is a surjective isometry and define
\[
U=(\Psi_M\circ J_M) \circ T \circ (\Psi_K\circ J_K)^{-1}: \Lipz([0,|K|])\to \Lipz([0,|M|]).
\]
By Proposition~\ref{prop:embed}, $\Psi_K\circ J_K$ and $\Psi_M\circ J_M$
are surjective linear isometries onto $\Lipz([0,|K|])$ and $\Lipz([0,|M|])$, respectively; hence $U$ is a well-defined surjective linear isometry.
Since $\Lipz([0,|K|])$ and $\Lipz([0,|M|])$ are dense in $\Czo{|K|}$ and $\Czo{|M|}$, respectively, $U$ extends uniquely to a surjective linear isometry
\[
\widehat U: \Czo{|K|} \to \Czo{|M|}.
\]
By Lemma~\ref{lem:banach_stone}, there exist $\varepsilon\in\{\pm 1\}$ and a homeomorphism
$\psi:[0,|M|]\to[0,|K|]$ with $\psi(0)=0$ such that
\[
\widehat U (G) =\varepsilon\, G \circ \psi \qquad\text{for all }G\in C_0([0,|K|]).
\]
Take $G(t)=t\in\Lipz([0,|K|])$. Since $G$ lies in the domain of $U$, we have $U(G)=\widehat U(G)$, and hence
\[
U(G)=\varepsilon(G\circ\psi)=\varepsilon\,\psi.
\]
Therefore $\psi\in \Lipz([0,|M|])$. Applying the same argument to $U^{-1}$ shows that $\psi^{-1}$ is Lipschitz.
Thus $\psi$ is a lipeomorphism. Since $\psi$ is a homeomorphism $[0,|M|]\to[0,|K|]$ with $\psi(0)=0$,
it must be increasing.

Let $f\in L^\infty(K)$ and set $F= J_K f$, $G= \Psi_K(F)=F\circ\sigma_K\in\Lipz([0,|K|])$.
Then
\[
(\Psi_M \circ J_M)(Tf)\ =\ U(G)\ =\ \varepsilon\,G\circ\psi.
\]
Applying $\Phi_M$ (i.e.\ composing with $\pi_M$) and using $\Phi_M\circ\Psi_M=\mathrm{id}_{X_M}$,
$\pi_M\circ\sigma_M=\mathrm{id}_{[0,|M|]}$, and $G=J_K f\circ\sigma_K$, we obtain on $M$:
\begin{align*}
J_M(Tf) & = \Phi_M(\Psi_M J_M(Tf)) \\
& = \varepsilon\,(G\circ\psi)\circ\pi_M \\
& =\varepsilon\,(J_K f)\circ \sigma_K\circ \psi\circ \pi_M \\
& =\varepsilon\,(J_K f)\circ \phi,
\end{align*}
which is \eqref{thm_J_represent} with $\phi$ given by \eqref{phi}.

For the pointwise formula \eqref{thm_T_represent}, fix $f\in L^\infty(K)$ and set
\[
Q\coloneqq \varepsilon\,(J_K f)\circ\phi \quad\text{on }M.
\]
To justify differentiation, we extend $J_M(Tf)$ and $Q$ to the whole interval $[\min M,\max M]$.

\smallskip
\textbf{Step 1: Auxiliary extensions.}
Define the interval extension of $J_M$ by
\[
\widetilde J_M(g)(\tilde{y})\coloneqq \int_{[\min M,\tilde{y}]\cap M} g,
\qquad \tilde{y}\in[\min M,\max M],
\]
and set
\[
H(\tilde{y})\coloneqq \widetilde J_M(Tf)(\tilde{y}),\qquad \tilde{y}\in[\min M,\max M].
\]
Then $H$ is absolutely continuous on $[\min M,\max M]$, and by the fundamental theorem of calculus for
Lebesgue integrals,
\begin{equation}\label{eq:Hprime}
H'(\tilde{y})=(Tf)(\tilde{y})\,\1_M(\tilde{y})\qquad\text{for a.e.\ }\tilde{y}\in[\min M,\max M].
\end{equation}

Next, set $\widetilde G \coloneqq \Psi_K(J_K f)\in\Lipz([0,|K|])$ and define for $\tilde y\in[\min M,\max M]$
\[
\bar Q(\tilde y)\coloneqq \varepsilon\,(J_K f)(\sigma_K(\psi(\Pi_M(\tilde y))))
=\varepsilon\,\widetilde G(\psi(\Pi_M(\tilde y))).
\]
Since $\Pi_M$ and $\psi$ are Lipschitz and $\widetilde G$ is Lipschitz, $\bar Q$ is absolutely continuous on
$[\min M,\max M]$.

\smallskip
\textbf{Step 2: Identification $\bm{H=\bar Q}$ on the interval.} On $M$ we have $\Pi_M=\pi_M$, hence $\bar Q|_M=Q$.
Also, by \eqref{thm_J_represent} and the definition of $\widetilde J_M$, we have $H|_M=J_M(Tf)=Q$.
Thus $H=\bar Q$ on $M$.
\par 
Finally, both $H$ and $\bar Q$ are constant on each gap of $M$ (because $H$ depends only on
$[\min M,\tilde{y}]\cap M$, and $\Pi_M$ is constant on gaps), so equality on $M$ implies equality on the whole interval:
\[
H=\bar Q\qquad\text{on }[\min M,\max M].
\]
Consequently,
\begin{equation}\label{eq:HprimeQprime}
H'=\bar Q'\qquad\text{a.e.\ on }[\min M,\max M].
\end{equation}

\smallskip
\textbf{Step 3: Differentiation of $\bm{\bar Q}$ and restriction to $\bm{M}$.}
Using \eqref{deriv_Psi_J}, we have
$\widetilde G'(t)=f(\sigma_K(t))$ for a.e.\ $t\in[0,|K|]$. Since
$\bar Q(t)=\varepsilon\,\widetilde G(\psi(\Pi_M(t)))$, the chain rule for absolutely continuous maps
and $\Pi_M'=\1_M$ a.e.\ on $[\min M, \max M]$ yield, for a.e.\ $\tilde{y}\in[\min M,\max M]$,
\begin{align*}
\bar Q'(\tilde{y})&=\varepsilon\,\widetilde G'\big(\psi(\Pi_M(\tilde{y}))\big)\,\psi'(\Pi_M(\tilde{y}))\,\Pi_M'(\tilde{y}) \\
& =\varepsilon\,f\big(\sigma_K(\psi(\Pi_M(\tilde{y})))\big)\,\psi'(\Pi_M(\tilde{y}))\,\1_M(\tilde{y}).
\end{align*}
Restricting to a.e.\ $y\in M$ (so $\Pi_M(y)=\pi_M(y)$ and $\1_M(y)=1$) gives
\begin{equation}\label{eq:Qprime_on_M}
\bar Q'(y)=\varepsilon\,(f\circ\phi)(y)\,(\psi'\circ\pi_M)(y)\qquad\text{for a.e.\ }y\in M.
\end{equation}
Comparing \eqref{eq:Hprime}--\eqref{eq:HprimeQprime} with \eqref{eq:Qprime_on_M} yields
\[
(Tf)(y)=\varepsilon\,(f\circ\phi)(y)\,(\psi'\circ\pi_M)(y)
\qquad\text{for a.e.\ }y\in M,
\]
which is \eqref{thm_T_represent}.

\medskip
Conversely, fix $\varepsilon\in\{\pm1\}$ and an increasing lipeomorphism
$\psi:[0,|M|]\to[0,|K|]$, and define $T$ by \eqref{thm_T_represent} with
$\phi= \sigma_K\circ\psi\circ\pi_M$. Since $\psi'\in L^\infty([0,|M|])$,
we have $Tf\in L^\infty(M)$ for every $f\in L^\infty(K)$.

Let $f\in L^\infty(K)$ be arbitrary. Consider the auxiliary interval functions
$H$ and $\bar Q$ defined in Steps~1--2 above. Then $H$ and $\bar Q$ are absolutely continuous, and the computation of Step~3 yields $H'=\bar Q'$ a.e.\ on
$[\min M,\max M]$. Since $H(\min M)=\bar Q(\min M)=0$, it follows that
$H=\bar Q$ on $[\min M,\max M]$. Restricting to $M$ gives
\[
J_M(Tf)(y)=H(y)=\bar Q(y)=\varepsilon\,(J_K f)(\phi(y))
\qquad (y\in M),
\]
which establishes \eqref{thm_J_represent}. Using $\widetilde G = \Psi_K(J_K f)=(J_K f)\circ\sigma_K$
we have, for $y\in M$,
\[
J_M(Tf)(y)=\varepsilon\,(J_K f)(\sigma_K(\psi(\pi_M(y))))
=\varepsilon\,\widetilde G(\psi(\pi_M(y))),
\]
which, due to the surjectivity of $\pi_M$ and $\psi$, implies
\[
\|Tf\|_{\mathrm A}=\sup_{y\in M}|J_M(Tf)(y)|
=\sup_{t\in[0,|K|]}|\widetilde G(t)|.
\]
Since $J_K f=\widetilde G\circ\pi_K$ on $K$ and $\pi_K(K)=[0,|K|]$, it follows that
\[
\sup_{t\in[0,|K|]}|\widetilde G(t)|=\sup_{x\in K}|J_K f(x)|=\|f\|_{\mathrm A}.
\]
Therefore $\|Tf\|_{\mathrm A}=\|f\|_{\mathrm A}$, so $T$ is an isometry. That $T$ is linear is clear from the representation formula~\eqref{thm_T_represent}.
\par
Surjectivity follows by symmetry: applying the same construction with $K,M$
interchanged and $\psi^{-1}$ in place of $\psi$ yields an inverse isometry.
\end{proof}

\begin{remark}\label{rem:isometric_consequence}
As a direct consequence of Theorem~\ref{thm:ABS}, if $K,M\subseteq\R$ are compact sets of positive Lebesgue measure, then $(L^\infty(K),\normA{\cdot})$ and $(L^\infty(M),\normA{\cdot})$ are linearly isometric. The subsequent sections address the finer geometric question of when the interval-level change of variables lifts to a homeomorphism or a lipeomorphism between the underlying compact sets.
\end{remark}

\section{Fiber \texorpdfstring{$\psi$}{\unichar{"03C8}}-compatibility and homeomorphic \texorpdfstring{$\phi$}{\unichar{"03D5}}} \label{sec:fiber_compat}

In this section we determine when an increasing lipeomorphism $\psi:[0,|M|]\to[0,|K|]$ lifts to an increasing homeomorphism $\phi:M\to K$ satisfying $\pi_K\circ\phi=\psi\circ\pi_M$. To this end, we introduce \emph{fiber $\psi$-compatibility}, which encodes the necessary and sufficient matching of exceptional fibers of $\pi_M$ and $\pi_K$. The proof of sufficiency constructs a lift first on a convenient dense subset $S_M\subseteq M$ (defined below) and then extends it to all of $M$ by monotonicity.

Let $\Q$ denote the rationals. Define
\begin{equation} \label{set_Q_M}
Q_M \coloneqq \big(\Q \cap [0, |M|]\big) \cup \{|M|\} \cup E_M 
\end{equation}
and 
\begin{equation} \label{set_S_M}
S_M \coloneqq \pi_M^{-1}(Q_M).
\end{equation}

\begin{lemma} \label{lem:S_M_dense}
The set $S_M$ is a dense subset of $M$.
\end{lemma}

\begin{proof}
Let $I\subseteq\R$ be an open interval with $I\cap M\neq\emptyset$. We show that
$I\cap S_M\neq\emptyset$.

If $I\cap M$ is a singleton $\{c\}$, then $c$ is an isolated point of $M$. Since $|M|>0$,
the set $M$ contains a point different from $c$, hence $c$ is an endpoint of at least
one gap. In particular, the fiber $\pi_M^{-1}(\pi_M(c))$ contains at least two points,
so $\pi_M(c)\in E_M\subseteq Q_M$, and therefore $c\in S_M$.

Assume now that $I\cap M$ contains at least two points, and choose $a<b$ in $I\cap M$.
If $\pi_M(a)=\pi_M(b)$, then $\pi_M(a)\in E_M\subseteq Q_M$ and hence $a\in I\cap S_M$.

Otherwise $\pi_M(a)<\pi_M(b)$. Choose $q\in(\Q\cap[0,|M|])\cap(\pi_M(a),\pi_M(b))\subseteq Q_M$.
Since $\Pi_M$ is continuous and nondecreasing on $[\min M,\max M]$ and $\Pi_M=\pi_M$ on $M$,
there exists $\tilde y\in[a,b]\subseteq I$ such that $\Pi_M(\tilde y)=q$.
If $\tilde y\in M$, then $\pi_M(\tilde y)=q\in Q_M$ and $\tilde y\in I\cap S_M$.
If $\tilde y\notin M$, then $\tilde y$ lies in a gap $(\ell,r)$ of $M$ contained in $(a,b)\subseteq I$,
and $\Pi_M(\ell)=\Pi_M(\tilde y)=q$ with $\ell\in M$. Thus $\ell\in I\cap S_M$.

In all cases $I\cap S_M\neq\emptyset$, hence $S_M$ is dense in $M$.
\end{proof}

\begin{definition}
Let $K,M\subseteq \R$ be compact sets of positive measure, and let $\psi:[0,|M|]\to[0,|K|]$ be an increasing lipeomorphism. The sets $K$ and $M$ are said to be \emph{fiber $\psi$-compatible} if 
\begin{equation}\label{eq:EK-match}
\psi(E_M)=E_K,
\end{equation}
and, for every $s\in E_M$, the fiber $\pi_M^{-1}(s)$ is order-isomorphic to the fiber $\pi_K^{-1}(\psi(s))$, i.e., there exists an increasing bijection 
\begin{equation} \label{eq:bijection_h_s}
h_s : \pi_M^{-1}(s) \to \pi_K^{-1}(\psi(s)).
\end{equation}
\end{definition}

\begin{example}[Failure of fiber compatibility: $E$-sets do not match]\label{ex:EK_mismatch}
Let
\[
K=[0,1],\qquad M=\left[ 0,\frac12 \right]\cup\{1\}.
\]
Then $|K|=1$ and $|M|=\frac12$. Moreover, $\pi_K$ is strictly increasing on $K$, hence every fiber is a
singleton and $E_K=\emptyset$.

On the other hand, $\pi_M(1/2)=\pi_M(1)=1/2$, so
\[
\pi_M^{-1}\left(\frac12\right)=\left\{\frac12,1\right\},
\qquad\text{and hence}\qquad
E_M=\left\{\frac12\right\}.
\]
Therefore, for any increasing lipeomorphism $\psi:[0,|M|]\to[0,|K|]$, we have
$\psi(E_M)=\{\psi(1/2)\}\neq\emptyset=E_K$, so $K$ and $M$ are not fiber $\psi$-compatible for any~$\psi$.
\end{example}

\begin{example}[Failure of fiber compatibility: fibers are not order-isomorphic]\label{ex:fiber_order_type}
Let
\[
K=[0,1]\cup\{2\},
\qquad
M=[0,1]\cup \{2\}\cup\{3\}.
\]
Then $|K|=|M|=1$ and one checks that
\[
E_K=E_M=\{1\}.
\]
Moreover,
\[
\pi_K^{-1}(1)=\{1,2\},
\qquad
\pi_M^{-1}(1)=\{1,2,3\}.
\]
Thus the fiber of $M$ over $1$ has three points whereas the corresponding fiber of $K$ has two points,
so there is no order-isomorphism between them. Hence $K$ and $M$ are not fiber $\psi$-compatible for any
increasing lipeomorphism $\psi:[0,|M|]\to[0,|K|]$ (indeed, necessarily $\psi(1)=1$).
\end{example}

\begin{theorem}\label{thm:homeo-lift}
Let $K,M\subseteq \R$ be compact sets of positive measure, and let $\psi:[0,|M|]\to[0,|K|]$ be an increasing lipeomorphism. There exists an increasing homeomorphism $\phi:M\to K$
satisfying
\begin{equation} \label{eq:pi_K_phi=psi_pi_M}
\pi_K\circ\phi=\psi\circ\pi_M
\end{equation}
if and only if $K$ and $M$ are fiber $\psi$-compatible. 
\end{theorem}

\begin{proof}
Assume first that there exists an increasing homeomorphism $\phi:M\to K$ satisfying~\eqref{eq:pi_K_phi=psi_pi_M}. Let $s \in E_M$. By definition, the fiber $\pi_M^{-1}(s)$ contains at least two distinct points, say $y_1 < y_2$. Since $\phi$ is injective, $\phi(y_1) \neq \phi(y_2)$. Condition \eqref{eq:pi_K_phi=psi_pi_M} implies
\[ \pi_K(\phi(y_1)) = \psi(\pi_M(y_1)) = \psi(s) = \psi(\pi_M(y_2)) = \pi_K(\phi(y_2)). \]
Thus, the fiber $\pi_K^{-1}(\psi(s))$ also contains at least two distinct points, so $\psi(s) \in E_K$. This implies $\psi(E_M) \subseteq E_K$.
Applying the same argument to $\phi^{-1}$ (using the relation $\pi_M \circ \phi^{-1} = \psi^{-1} \circ \pi_K$) shows that $\psi^{-1}(E_K) \subseteq E_M$, or equivalently $E_K \subseteq \psi(E_M)$.
Thus, $\psi(E_M) = E_K$, establishing condition \eqref{eq:EK-match}. \par 

Moreover, for each $s \in E_M$ we have $\phi(\pi_M^{-1}(s)) = \pi_K^{-1}(\psi(s))$. Indeed, if $y \in \pi_M^{-1}(s)$, i.e., $\pi_M(y) = s$, then by \eqref{eq:pi_K_phi=psi_pi_M} we deduce 
\[\pi_K(\phi(y)) = \psi(\pi_M(y)) = \psi(s), \]
and so $\phi(\pi_M^{-1}(s)) \subseteq \pi_K^{-1}(\psi(s))$. Conversely, if $x \in \pi_K^{-1}(\psi(s))$, since $\phi$ is surjective, there exists $y \in M$ such that $\phi(y)=x$. By \eqref{eq:pi_K_phi=psi_pi_M} we have 
\[\psi(\pi_M(y))=\pi_K(\phi(y))=\pi_K(x)=\psi(s), \]
and since $\psi$ is injective it follows that $\pi_M(y) = s$; hence $x = \phi(y) \in \phi(\pi_M^{-1}(s))$. Therefore, for each $s \in E_M$, the restriction of $\phi$ to the fiber $\pi_M^{-1}(s)$ is an increasing bijection onto $\pi_K^{-1}(\psi(s))$, and hence $\pi_M^{-1}(s)$ and $\pi_K^{-1}(\psi(s))$ are order-isomorphic. This proves \eqref{eq:bijection_h_s} and we conclude that $K$ and $M$ are fiber $\psi$-compatible.

\medskip
Assume now that $\psi$ is an increasing lipeomorphism and that $K$ and $M$ are fiber $\psi$-compatible. We construct $\phi$.

\smallskip
\textbf{Step 1: Dense subsets $\bm{S_M \subseteq M}$ and $\bm{S_K^\psi \subseteq K}$.}
Recall $Q_M$ and $S_M$ from \eqref{set_Q_M} and~\eqref{set_S_M}, 
\[
Q_M = (\Q\cap[0,|M|]) \cup \{|M|\} \cup E_M,\qquad S_M=\pi_M^{-1}(Q_M),
\]
and note that by Lemma~\ref{lem:S_M_dense}, $S_M$ is a dense subset of $M$. Define
\begin{equation} \label{set_Q_K_psi}
Q_K^\psi \coloneqq \psi(Q_M)\subseteq[0,|K|],
\end{equation}
and 
\begin{equation} \label{set_S_K_psi}
S_K^\psi \coloneqq \pi_K^{-1}(Q_K^\psi) \subseteq K.
\end{equation}

We claim that $S_K^\psi$ is dense in $K$.
Since $Q_M$ contains $\Q\cap[0,|M|]$ and $\psi$ is a homeomorphism, the set
$Q_K^\psi=\psi(Q_M)$ is dense in $[0,|K|]$.
Moreover, $E_K=\psi(E_M)\subseteq Q_K^\psi$ by~\eqref{eq:EK-match}.
Therefore, the argument of Lemma~\ref{lem:S_M_dense} (with $M$ replaced by $K$ and $Q_M$ replaced by $Q_K^\psi$)
shows that $\pi_K^{-1}(Q_K^\psi)=S_K^\psi$ is dense in $K$.

\smallskip
\textbf{Step 2: An order-preserving bijection $\bm{\phi_0: S_M \to S_K^\psi}$.} 
Let $y \in S_M$ and set $s = \pi_M(y)$. 

On the one hand, if $s \notin E_M$, then $s\in Q_M\setminus E_M$ and the fiber $\pi_M^{-1}(s)$ is the singleton~$\{y\}$. By \eqref{eq:EK-match}, $\psi(s) \notin E_K$, hence $\pi_K^{-1}(\psi(s))$ is also a singleton, say $\{x\}$. Since $s\in Q_M$, we have $\psi(s)\in Q_K^\psi$, so $x\in S_K^\psi$.
Define
\[
\phi_0(y)\coloneqq x.
\]
Then
\[
\pi_K(\phi_0(y))=\psi(s)=\psi(\pi_M(y)).
\]

On the other hand, if $s\in E_M$, then by the fiber $\psi$-compatibility, there exists an increasing bijection
\[
h_s : \pi_M^{-1}(s) \to \pi_K^{-1}(\psi(s)).
\]
Moreover $s\in Q_M$ and so $\psi(s)\in Q_K^\psi$.
We define for every $y \in \pi_M^{-1}(s)$,
\[
\phi_0(y)\coloneqq h_s(y).
\]
Then, for every $y\in\pi_M^{-1}(s)$,
\[
\pi_K(\phi_0(y))=\psi(s)=\psi\big(\pi_M(y)\big),
\]
and $\phi_0(y)\in S_K^\psi$.
\par
This defines $\phi_0$ on all of $S_M$, with values in $S_K^\psi$. By construction, for each $s\in Q_M$, the map $\phi_0$ induces a bijection
\[
\phi_0:\ \pi_M^{-1}(s)\to \pi_K^{-1}(\psi(s)).
\]
Since $\psi|_{Q_M}:Q_M\to Q_K^\psi$ is a bijection, it follows that $\phi_0:S_M\to S_K^\psi$ is a bijection.
Indeed, for every $y\in S_M$ we have $\pi_K(\phi_0(y))=\psi(\pi_M(y))$, and hence, if $\phi_0(y_1)=\phi_0(y_2)$ then
$\psi(\pi_M(y_1))=\psi(\pi_M(y_2))$; since $\psi$ is injective, $\pi_M(y_1)=\pi_M(y_2)=:s$.
Thus $y_1,y_2\in \pi_M^{-1}(s)$ and the fiberwise injectivity yields $y_1=y_2$.
Surjectivity follows similarly: given $x\in S_K^\psi$, let $t=\pi_K(x)\in Q_K^\psi$ and choose the unique $s\in Q_M$
with $\psi(s)=t=\pi_K(x)$; since $\phi_0:\pi_M^{-1}(s)\to\pi_K^{-1}(\psi(s))$ is surjective, there exists $y\in\pi_M^{-1}(s)\subseteq S_M$
such that $\phi_0(y)=x$.

Moreover, $\phi_0$ is strictly increasing on $S_M$. Indeed, let $y_1<y_2$ in $S_M$. If $\pi_M(y_1)=\pi_M(y_2)=s\in E_M$, then $y_1,y_2$ lie in the same fiber
$\pi_M^{-1}(s)$ and, by construction, $\phi_0$ restricts to the strictly increasing map $h_s$ on this fiber, so
\[
\phi_0(y_1) = h_s(y_1) < h_s(y_2) = \phi_0(y_2).
\]
If instead $\pi_M(y_1)\neq\pi_M(y_2)$, then by monotonicity of $\pi_M$ we have
\[
\pi_M(y_1)<\pi_M(y_2).
\]
Since $\psi$ is strictly increasing,
\[
\psi(\pi_M(y_1))<\psi(\pi_M(y_2)).
\]
Then, using the identity $\pi_K(\phi_0(y_i))=\psi(\pi_M(y_i))$, $i = 1,2$, we obtain
\[
\pi_K(\phi_0(y_1))<\pi_K(\phi_0(y_2)),
\]
which, by the monotonicity of $\pi_K$, forces $\phi_0(y_1)<\phi_0(y_2)$, proving the claim.

\smallskip
\textbf{Step 3: Extending $\bm{\phi_0}$ to an increasing bijection $\bm{\phi:M\to K}$.}
For $y\in M$ define
\[
\phi(y)\coloneqq
\sup\big\{\phi_0(z) \ | \ z\in S_M,\ z\le y\big\}.
\]
Clearly $\phi(M) \subseteq K$ and $\phi$ is increasing by construction.

If $y\in S_M$, then the order-preserving property of $\phi_0$ implies
\[
\sup\big\{ \phi_0(z) \  | \ z\in S_M,\ z\le y\big\}=\phi_0(y),
\]
so $\phi$ extends $\phi_0$, that is, $\phi(y)=\phi_0(y)$ for all $y\in S_M.$
\par
A symmetric construction, starting from the inverse bijection $\phi_0^{-1}:S_K^\psi\to S_M$, yields an increasing map
\[
\tilde{\phi}:K\to M
\]
extending $\phi_0^{-1}$. \par
We now prove that $\tilde{\phi}\circ\phi=\mathrm{id}_M$ and
$\phi\circ\tilde{\phi}=\mathrm{id}_K$. Fix $y\in M$. Since $S_M$ is dense in $M$, there exist sequences
$\{z_n^-\},\{z_n^+\}\subseteq S_M$ with
\[
z_n^- \le y \le z_n^+,\qquad
z_n^- \uparrow y,\quad z_n^+ \downarrow y.
\]
Because $\phi$ is increasing, we have
\[
\phi_0(z_n^-)\leq \phi(y)\leq \phi_0(z_n^+)
\]
for every $n$. Applying the increasing map $\tilde{\phi}$ to this inequality yields
\begin{equation}\label{eq:psi0-brackets}
\tilde{\phi}(\phi_0(z_n^-)) \leq  \tilde{\phi}(\phi(y)) \leq \tilde{\phi}(\phi_0(z_n^+)).
\end{equation}
By construction $\tilde{\phi}$ extends $\phi_0^{-1}$, so
\[
\tilde{\phi}(\phi_0(z_n^-)) = z_n^-,\qquad
\tilde{\phi}(\phi_0(z_n^+)) = z_n^+.
\]
Thus \eqref{eq:psi0-brackets} becomes
\[
z_n^- \leq \tilde{\phi}(\phi(y)) \leq  z_n^+ \qquad (n\in\mathbb N).
\]
Letting $n\to\infty$ and using $z_n^- \uparrow y$, $z_n^+ \downarrow y$ gives
\[
y  \leq \tilde{\phi}(\phi(y)) \leq y,
\]
hence $\tilde{\phi}(\phi(y))=y$. Since $y\in M$ is arbitrary, we have
\[
\tilde{\phi}\circ\phi=\mathrm{id}_M.
\]
A completely symmetric argument shows that $\phi\circ\tilde{\phi}=\mathrm{id}_K$. Therefore $\phi$ and $\tilde{\phi}$ are mutual inverses, so $\phi:M\to K$ is a bijection.

\smallskip
\textbf{Step 4: $\bm{\phi}$ is a homeomorphism satisfying $\bm{\pi_K \circ \phi = \psi \circ \pi_M}$.}
First, we prove that $\phi$ is continuous. Let $y \in M$ and let $\{y_n\} \subseteq M$ be a sequence such that $y_n \uparrow y$. Since $\phi$ is increasing, the sequence $\{\phi(y_n)\}$ is increasing and bounded above by $\phi(y)$. Let $L = \lim_{n \to \infty} \phi(y_n)$. Since $K$ is closed, $L \in K$. Because $\phi$ is surjective, there exists $z \in M$ such that $\phi(z) = L$.\par
Since $\phi(y_n) \leq L$ for all $n$, we have $\phi(y_n) \leq \phi(z)$, which implies $y_n \leq z$ for all $n$, and taking the limit yields $y \leq z$.
Conversely, $L \le \phi(y)$ implies $\phi(z) \le \phi(y)$, so $z \le y$.
Thus $z=y$, which implies $\lim_{n \to \infty} \phi(y_n) = \phi(y)$.
The argument for $y_n \downarrow y$ is analogous. Therefore, $\phi$ is continuous. Since $\phi$ is a continuous bijection from a compact space $M$ to a Hausdorff space $K$, it is a homeomorphism. \par 

Finally, we verify identity \eqref{eq:pi_K_phi=psi_pi_M}. By construction, for every $y\in S_M$ we have
\[
\pi_K\big(\phi(y)\big)
=\pi_K\big(\phi_0(y)\big)
=\psi\big(\pi_M(y)\big).
\]
The functions $\pi_K\circ\phi$ and $\psi\circ\pi_M$ are continuous on $M$. Since they agree on the dense subset $S_M$, they must agree everywhere:
\[
\pi_K\circ\phi=\psi\circ\pi_M\quad\text{on }M.
\]
This proves the existence of the required homeomorphism. Together with the necessity part, the theorem is proved.
\end{proof}

\begin{lemma}\label{lem:ae-equivalence}
Let $K,M \subseteq \R$ be compact sets of positive measure, and assume that $K$ and $M$ are fiber $\psi$-compatible for some increasing lipeomorphism $\psi:[0,|M|] \to [0,|K|]$. Let $\phi:M\to K$ be an increasing homeomorphism such that $\pi_K\circ\phi=\psi\circ\pi_M$, and set
\begin{equation*} 
\phi_\sigma \coloneqq \sigma_K \circ \psi \circ \pi_M.
\end{equation*}
Then $\phi=\phi_\sigma$ a.e.\ on $M$.
\end{lemma}

\begin{proof}
We show that the set 
\[
\mathcal{N} = \big\{y \in M \ | \ \phi(y) \neq \phi_\sigma(y)\big\}
\]
has measure zero. To that end, we prove that 
\begin{equation} \label{eq:aux1}
\mathcal{N} = \bigcup_{s\in E_M}\Big(\pi_M^{-1}(s)\setminus\big\{\min \pi_M^{-1}(s) \big\}\Big).
\end{equation}
Then the result follows because $E_M$ is countable and, for each $s\in E_M$,
the fiber $\pi_M^{-1}(s)$ has measure $0$ (indeed, $\pi_M^{-1}(s)=M\cap I_s$ where
$I_s=\{\Pi_M=s\}$ is a closed interval and $\lambda(M\cap I_s)=\Pi_M(\max I_s)-\Pi_M(\min I_s)=s-s=0$).
Hence the right-hand side of~\eqref{eq:aux1} is a countable union of null sets.
\par 
First, we write $\mathcal{N} = (\mathcal{N} \cap M_1) \cup (\mathcal{N} \cap M_2)$, where
\[M_1 = \big\{ y \in M \ | \ \pi_M(y) \notin E_M \big\} \]
and
\[M_2 =  \big\{ y \in M \ | \ \pi_M(y) \in E_M \big\}.\]
\par 
We claim that $\mathcal{N} \cap M_1 = \emptyset$. Indeed, if $y \in M$ with $s = \pi_M(y) \notin E_M$, then $\pi_M^{-1}(s)$ is the singleton $\{y \}$. By \eqref{eq:EK-match}, $\psi(s) \notin E_K$, so $\pi_K^{-1}(\psi(s))$ is also a singleton, say $\{z\}$. According to \eqref{eq:pi_K_phi=psi_pi_M} we must have $z = \phi(y)$. Moreover, $\phi_\sigma(y) = \sigma_K(\psi(s))=\min\{z\} = z$. Thus, $\phi(y) = \phi_\sigma(y)$, so $y \notin \mathcal{N}$, establishing the claim. 
\par 
Now, we observe that $M_2 = \pi_M^{-1}(E_M)$ and hence
\[\mathcal{N} \cap M_2 = \bigcup_{s \in E_M} \mathcal{N} \cap \pi_M^{-1}(s). \]
Let $s \in E_M$. By Theorem~\ref{thm:homeo-lift}, the restriction of $\phi$ to $\pi_M^{-1}(s)$ is an order-isomorphism onto $\pi_K^{-1}(\psi(s))$. Since order-isomorphisms preserve minima, we must have
    \[
      \phi\big(\min \pi_M^{-1}(s)\big) = \min \pi_K^{-1}(\psi(s)).
    \]
    On the other hand, for any $y \in \pi_M^{-1}(s)$, the selector map is defined as
    \[
      \phi_\sigma(y) = \sigma_K(\psi(s)) = \min \pi_K^{-1}(\psi(s)).
    \]
   Thus, for $y = \min \pi_M^{-1}(s)$, we have $\phi(y) = \min \pi_K^{-1}(\psi(s)) = \phi_\sigma(y)$; but for any $y \in \pi_M^{-1}(s) \setminus \{\min \pi_M^{-1}(s)\}$, the strict monotonicity of $\phi$ implies
          $\phi(y) > \phi(\min \pi_M^{-1}(s)) = \min \pi_K^{-1}(\psi(s)) = \phi_\sigma(y)$, so $\phi(y) \neq \phi_\sigma(y)$. \par 
          Therefore, on each fiber over $s\in E_M$, the two maps differ precisely at the points strictly greater than the minimum:
    \[
      \mathcal{N} \cap \pi_M^{-1}(s) = \pi_M^{-1}(s)\setminus\big\{\min \pi_M^{-1}(s)\big\},
    \]
 which establishes \eqref{eq:aux1} and finishes the proof.
\end{proof}

\begin{corollary}\label{cor:fiber_compatible_isometry}
Let $K,M \subseteq \R$ be compact sets of positive measure. Assume that there exists an increasing lipeomorphism $\psi:[0,|M|]\to[0,|K|]$ such that $K$ and $M$ are fiber $\psi$-compatible, and let $T:(L^\infty(K),\normA{\cdot})\to (L^\infty(M),\normA{\cdot})$ be the surjective linear isometry associated to $\psi$ as in Theorem~\ref{thm:ABS}. Then $T$ admits the representation
\[
(Tf)(y)=\varepsilon\, f(\phi(y))\,\psi'(\pi_M(y))
\quad\text{for a.e.\ }y\in M,
\]
where $\varepsilon \in \{\pm 1 \}$ and $\phi:M\to K$ is an increasing homeomorphism satisfying $\pi_K\circ\phi=\psi\circ\pi_M$.
\end{corollary}

\begin{proof}
Since $K$ and $M$ are fiber $\psi$-compatible, Theorem~\ref{thm:homeo-lift} guarantees the existence of an increasing homeomorphism $\phi: M \to K$ satisfying $\pi_K \circ \phi = \psi \circ \pi_M$.
By Theorem~\ref{thm:ABS}, the isometry $T$ admits the representation
\[
(Tf)(y)=\varepsilon\, f(\phi_\sigma(y))\,\psi'(\pi_M(y))\quad\text{for a.e.\ }y\in M,
\]
where $\phi_\sigma=\sigma_K\circ\psi\circ\pi_M$.
Finally, by Lemma~\ref{lem:ae-equivalence}, we have $\phi = \phi_\sigma$ a.e.\ on $M$. Thus, we may replace $\phi_\sigma$ with the homeomorphism $\phi$ in the representation formula.
\end{proof}

\section{Gap \texorpdfstring{$\psi$}{\unichar{"03C8}}-compatibility and lipeomorphic \texorpdfstring{$\phi$}{\unichar{"03D5}}} \label{sec:gap_compat}

In this section we determine when a homeomorphism $\phi$ obtained in Theorem~\ref{thm:homeo-lift} is a lipeomorphism. This requires a geometric condition stronger than fiber $\psi$-compatibility, controlling the relative scaling of the gaps. The key observation is that for $a<b$ with $a,b\in K$ the length $b-a$ decomposes into
the measure increment $\pi_K(b)-\pi_K(a)$ plus the total length of gaps inside $(a,b)$.
Gap $\psi$-compatibility provides uniform control of the latter term, yielding
bi-Lipschitz bounds for $\phi$. Conditions controlling the scaling of gaps are reminiscent of the literature on Lipschitz equivalence of Cantor-type and self-similar subsets of $\R$; see, for instance, \cite{falconer1992lipschitz,david1997fractured,rao2006lipschitz} (and the survey \cite{rao2013lipschitz} for a broader account).
\par
First, we establish that any local order-isomorphism between fibers induces a well-defined correspondence between the incident gaps.

\begin{lemma} \label{lem:existence_unique_gap}
Let $K,M\subseteq\R$ be compact sets of positive measure, and let $\psi:[0,|M|]\to[0,|K|]$ be an increasing lipeomorphism. Assume that $K$ and $M$ are fiber $\psi$-compatible. Let $U=(\ell',r')$ be a gap of $M$ and let $s = \pi_M(\ell')=\pi_M(r')$.
For any order-isomorphism $h: \pi_M^{-1}(s) \to \pi_K^{-1}(\psi(s))$, there exists a unique gap $V=(\ell,r)$ of $K$ such that
\begin{equation} \label{eq:endpoint_mapping}
\ell = h(\ell') \quad \text{and} \quad r = h(r').
\end{equation}
We call $V$ the \emph{corresponding gap to $U$ via $h$}, denoted by $V_{U,h}$.
\end{lemma}

\begin{proof}
Since $U=(\ell',r')$ is a gap of $M$, we have $(\ell',r')\cap M=\emptyset$. In particular,
$\ell',r'\in M$ and $\pi_M(\ell')=\pi_M(r')=s$, and there is no $y\in \pi_M^{-1}(s)$ with
$\ell'<y<r'$, i.e.\ $\ell'$ and $r'$ are consecutive elements of the fiber $\pi_M^{-1}(s)$.
Since $h$ is an order-isomorphism, it preserves adjacency, so
\[
\ell \coloneqq h(\ell') \quad\text{and}\quad r \coloneqq h(r')
\]
are consecutive elements of the fiber $\pi_K^{-1}(\psi(s))$; in particular $\ell<r$.

Moreover, $\Pi_K$ is nondecreasing on $[\min K,\max K]$ and
\[
\Pi_K(\ell)=\pi_K(\ell)=\psi(s)=\pi_K(r)=\Pi_K(r),
\]
so $\Pi_K$ is constant on $[\ell,r]$. Hence, for every $x\in(\ell,r)$ we have
$\Pi_K(x)=\Pi_K(\ell)=\psi(s)$. If there existed $x\in K\cap(\ell,r)$, then
$\pi_K(x)=\Pi_K(x)=\psi(s)$, so $x\in \pi_K^{-1}(\psi(s))$, contradicting that $\ell<r$
are consecutive in that fiber. Therefore $(\ell,r)\cap K=\emptyset$, and $V=(\ell,r)$ is a gap of $K$. \par
 Uniqueness is immediate, since $\ell$ and $r$ are uniquely determined by $h$.
\end{proof}

\begin{definition}\label{def:gap_compatibility}
Let $K,M \subseteq \R$ be compact sets of positive measure, and let
$\psi:[0,|M|]\to[0,|K|]$ be an increasing lipeomorphism.
We say that $K$ and $M$ are \emph{gap $\psi$-compatible} if $K$ and $M$ are fiber $\psi$-compatible and
there exist order-isomorphisms
\[
h_s:\pi_M^{-1}(s)\to \pi_K^{-1}(\psi(s)) \qquad (s\in E_M)
\]
and a constant $C\ge 1$ such that for every gap $U=(\ell',r')$ of $M$ with
$s=\pi_M(\ell')=\pi_M(r')$ one has
\begin{equation}\label{eq:gap_control}
C^{-1}\le \frac{\lambda(V_{U,h_s})}{\lambda(U)}\le C,
\end{equation}
where $V_{U,h_s}$ is the corresponding gap of $K$ via $h_s$.
\end{definition}

\begin{example}[Fiber-compatible but not gap-compatible]\label{ex:fiber_not_gap}
Let
\[
K=\{0\}\cup\Bigl\{\frac1n : n\in\N\Bigr\}\cup[1,2],
\qquad
M=\{0\}\cup\Bigl\{\frac1{n^2} : n\in\N\Bigr\}\cup[1,2].
\]
Then $K$ and $M$ are compact, $|K|=|M|=1$, and we take $\psi=\mathrm{id}_{[0,1]}$.

\smallskip
\noindent\emph{Fiber $\psi$-compatibility.}
Since the only positive-measure part is $[1,2]$, we have
\[
\Pi_K(x)=\Pi_M(x)=
\begin{cases}
0, & x<1,\\
x-1, & x\in[1,2],
\end{cases}
\]
and hence, for $t\in(0,1]$,
\[
\pi_K^{-1}(t)=\{1+t\}=\pi_M^{-1}(t).
\]
The only nontrivial fiber occurs at $t=0$, and
\[
\pi_K^{-1}(0)=\{0\}\cup\Bigl\{\frac1n : n\in\N\Bigr\},
\qquad
\pi_M^{-1}(0)=\{0\}\cup\Bigl\{\frac1{n^2} : n\in\N\Bigr\}.
\]
In particular, $E_K=E_M=\{0\}$ and $\psi(E_M)=E_K$.
Define an order-isomorphism $h_0:\pi_M^{-1}(0)\to\pi_K^{-1}(0)$ by
\[
h_0(0)=0,\qquad h_0\Bigl(\frac1{n^2}\Bigr)=\frac1n\quad(n\in\N).
\]
Thus $K$ and $M$ are fiber $\psi$-compatible.

\smallskip
\noindent\emph{Failure of gap $\psi$-compatibility.}
For $n\in \N$, let
\[
U_n=\left(\frac1{(n+1)^2},\frac1{n^2}\right)\!,
\]
which is a gap of $M$. The corresponding gap of $K$ via $h_0$ is
\[
V_n=V_{U_n,h_0}=\left(\frac1{n+1},\frac1n\right)\!.
\]
Therefore,
\[
\lambda(V_n)=\frac1n-\frac1{n+1}=\frac1{n(n+1)},
\qquad
\lambda(U_n)=\frac1{n^2}-\frac1{(n+1)^2}=\frac{2n+1}{n^2(n+1)^2},
\]
and hence
\[
\frac{\lambda(V_{U_n,h_0})}{\lambda(U_n)}
=\frac{\lambda(V_n)}{\lambda(U_n)}
=\frac{n(n+1)}{2n+1}\to\infty, \quad \text{as} \ n\to\infty.
\]
Thus there is no constant $C\ge1$ satisfying \eqref{eq:gap_control}, and $K$ and $M$ are not gap
$\psi$-compatible.
\end{example}

\begin{lemma}\label{lem:interval_decomposition}
Let $K\subseteq\R$ be compact. Then for every $a<b$ with $a,b \in K$ one has
\begin{equation}\label{eq:decomp_interval}
b-a = \big(\pi_K(b)-\pi_K(a)\big)+\!\!\sum_{V\subseteq(a,b)}\!\! \lambda(V),
\end{equation}
where the sum runs over all gaps $V$ of $K$ contained in $(a,b)$.
\end{lemma}

\begin{proof}
Since $K$ is compact (hence closed) and $a,b \in K$, each connected component of $(a,b) \setminus K$ is an open interval $(\ell,r) \subseteq (a,b)$ whose endpoints $\ell,r$ belong to $K$. Hence these (at most countably many) components are exactly the gaps of $K$ contained in~$(a,b)$, and so
\[
(a,b)=\big((a,b)\cap K\big)\cup\bigcup_{V\subseteq(a,b)} V.
\]
Therefore
\[
b-a=\lambda\big((a,b)\cap K\big)+\sum_{V\subseteq(a,b)}\lambda(V).
\]
By definition of $\pi_K$, we have $\lambda\big((a,b)\cap K\big)=\pi_K(b)-\pi_K(a)$, which yields
\eqref{eq:decomp_interval}.
\end{proof}

\begin{theorem} \label{thm:lipeomorphism}
Let $K,M\subseteq \R$ be compact sets of positive measure, and let $\psi:[0,|M|]\to[0,|K|]$ be an increasing lipeomorphism. There exists a lipeomorphism $\phi: M \to K$ satisfying $\pi_K \circ \phi = \psi \circ \pi_M$ if and only if $K$ and $M$ are gap $\psi$-compatible.
\end{theorem}

\begin{proof}
Assume first that $K$ and $M$ are gap $\psi$-compatible, and let $\phi:M\to K$ be the increasing
homeomorphism given by Theorem~\ref{thm:homeo-lift}, constructed from the family $\{h_s\}_{s\in E_M}$.
We prove that $\phi$ is Lipschitz.

Let $x<y$ in $M$. Since $\phi(x),\phi(y)\in K$, applying Lemma~\ref{lem:interval_decomposition} to $K$
on the interval $(\phi(x),\phi(y))$ gives
\begin{equation}\label{eq:phi-decomp}
\phi(y)-\phi(x)
=\big(\pi_K(\phi(y))-\pi_K(\phi(x))\big)+\sum_{V\subseteq(\phi(x),\phi(y))}\lambda(V),
\end{equation}
where the sum runs over all gaps $V$ of $K$ contained in $(\phi(x),\phi(y))$.
Since $\pi_K\circ\phi=\psi\circ\pi_M$, we have
\begin{align*}
\pi_K(\phi(y))-\pi_K(\phi(x))
& \le \|\psi'\|_\infty\,(\pi_M(y)-\pi_M(x)) \\
& \le \|\psi'\|_\infty\,(y-x).
\end{align*}
Next, we estimate the gap sum. Let $U=(\ell',r')$ be a gap of $M$ contained in $(x,y)$ and set
$s=\pi_M(\ell')=\pi_M(r')$. Then Lemma~\ref{lem:existence_unique_gap} implies that
$V_{U,h_s}=(h_s(\ell'),h_s(r'))$ is a gap of $K$, and by construction of $\phi$ we have
$\phi(\ell')=h_s(\ell')$ and $\phi(r')=h_s(r')$. In particular, the gap of $K$ with endpoints
$\phi(\ell')$ and $\phi(r')$ is precisely $V_{U,h_s}$. Applying the same argument to $\phi^{-1}$,
we obtain that the correspondence
\[
(\ell',r')\mapsto (\phi(\ell'),\phi(r'))
\]
is a bijection between the gaps of $M$ contained in $(x,y)$ and the gaps of $K$ contained in
$(\phi(x),\phi(y))$. 
Therefore
\[
\sum_{V\subseteq(\phi(x),\phi(y))}\lambda(V)
=\sum_{U\subseteq(x,y)}\lambda(V_{U,h_s})
\le C \sum_{U\subseteq(x,y)}\lambda(U),
\]
where $C$ is the constant from \eqref{eq:gap_control}.

Since the gaps $U\subseteq(x,y)$ are pairwise disjoint open intervals contained in $(x,y)$, we have
\[\sum_{U\subseteq(x,y)}\lambda(U)\le y-x,\] and hence
\[
\sum_{V\subseteq(\phi(x),\phi(y))}\lambda(V)\le C\,(y-x).
\]
Combining this with \eqref{eq:phi-decomp} yields
\[
\phi(y)-\phi(x)\le (\|\psi'\|_\infty+C)\,(y-x),
\]
so $\phi$ is Lipschitz. Applying the same argument to $\phi^{-1}$ (using $\psi^{-1}$ and the corresponding gap control) shows that $\phi^{-1}$ is Lipschitz, hence $\phi$ is a lipeomorphism.

\smallskip
Conversely, assume there exists a lipeomorphism $\phi:M\to K$ such that
$\pi_K\circ\phi=\psi\circ\pi_M$. For each $s\in E_M$, define $h_s$ as the restriction
$h_s\coloneqq \phi|_{\pi_M^{-1}(s)}$. By the necessity part of Theorem~\ref{thm:homeo-lift}, each $h_s$
is an order-isomorphism.

Let $U=(\ell',r')$ be a gap of $M$ and set $s=\pi_M(\ell')=\pi_M(r')$. Since $\phi$ is increasing,
$(\phi(\ell'),\phi(r'))$ is an open interval, and because $\phi$ is a homeomorphism $M\to K$ we have
$(\phi(\ell'),\phi(r'))\cap K=\emptyset$, so $V\coloneqq (\phi(\ell'),\phi(r')) = (h_s(\ell'),h_s(r'))$ is a gap of $K$.
By construction, $V=V_{U,h_s}$. If $\mathrm{L}$ and $\mathrm{L}'$ are Lipschitz constants for $\phi$ and $\phi^{-1}$, then
\[
\lambda(V_{U,h_s})=\phi(r')-\phi(\ell')\le\mathrm{L}\,(r'-\ell')=\mathrm{L}\,\lambda(U),
\]
and similarly $\lambda(U)\le \mathrm{L}'\,\lambda(V_{U,h_s})$. Therefore \eqref{eq:gap_control} holds with
$C=\max\{\mathrm{L},\mathrm{L}'\}$, and $K$ and $M$ are gap $\psi$-compatible.
\end{proof}

\begin{corollary}\label{cor:gap_compatible_isometry}
Let $K,M\subseteq\R$ be compact sets of positive measure. Assume that there exists an increasing
lipeomorphism $\psi:[0,|M|]\to[0,|K|]$ such that $K$ and $M$ are gap $\psi$-compatible, and let
$T:(L^\infty(K),\normA{\cdot})\to(L^\infty(M),\normA{\cdot})$ be the surjective linear
isometry associated to $\psi$ as in Theorem~\ref{thm:ABS}. Then there exist $\varepsilon\in\{\pm1\}$
and an increasing lipeomorphism $\phi:M\to K$ with $\pi_K\circ\phi=\psi\circ\pi_M$ such that
\[
(Tf)(y)=\varepsilon\,f(\phi(y))\,\tilde\phi'(y)\qquad\text{for a.e.\ }y\in M,
\]
where $\tilde\phi:[\min M,\max M]\to[\min K,\max K]$ denotes the extension of $\phi$ which is affine on
each gap of $M$.
\end{corollary}

\begin{proof}
Since $K$ and $M$ are gap $\psi$-compatible, Theorem~\ref{thm:lipeomorphism} yields an increasing
lipeomorphism $\phi:M\to K$ such that $\pi_K\circ\phi=\psi\circ\pi_M$ on $M$.

By Theorem~\ref{thm:ABS}, the isometry $T$ admits the representation
\[
(Tf)(y)=\varepsilon\,f(\phi_\sigma(y))\,\psi'(\pi_M(y))\qquad\text{for a.e.\ }y\in M,
\]
where $\phi_\sigma=\sigma_K\circ\psi\circ\pi_M$ and $\varepsilon\in\{\pm1\}$.
By Lemma~\ref{lem:ae-equivalence}, $\phi_\sigma=\phi$ a.e.\ on $M$, hence
\begin{equation}\label{eq:T_rep_with_psi_prime}
(Tf)(y)=\varepsilon\,f(\phi(y))\,\psi'(\pi_M(y))\qquad\text{for a.e.\ }y\in M.
\end{equation}
\par
Let $\tilde\phi:[\min M,\max M]\to[\min K,\max K]$ be defined by $\tilde\phi=\phi$ on $M$, and on each
gap $U=(\ell',r')$ of $M$ by linear interpolation:
\[
\tilde\phi(\tilde{y})\coloneqq \phi(\ell')+\frac{\tilde{y}-\ell'}{r'-\ell'}\big(\phi(r')-\phi(\ell')\big),
\qquad \tilde{y}\in[\ell',r'].
\]
Then $\tilde\phi$ is increasing and Lipschitz, hence differentiable a.e.\ on the interval $[\min M,\max M]$.

We claim that
\begin{equation}\label{eq:Pi_identity}
\Pi_K(\tilde\phi(\tilde{y}))=\psi(\Pi_M(\tilde{y}))\qquad\text{for all }\tilde{y}\in[\min M,\max M].
\end{equation}
If $y\in M$, then $\Pi_M(y)=\pi_M(y)$ and $\tilde\phi(y)=\phi(y)\in K$, hence
\[
\Pi_K(\tilde\phi(y))=\pi_K(\phi(y))=\psi(\pi_M(y))=\psi(\Pi_M(y)).
\]

Now let $\tilde{y}\in(\ell',r')$ where $U=(\ell',r')$ is a gap of $M$, and set
$s\coloneqq \Pi_M(\ell')=\Pi_M(r')$ (so $\Pi_M\equiv s$ on $[\ell',r']$).
By gap $\psi$-compatibility and Lemma~\ref{lem:existence_unique_gap}, $\phi$ maps the endpoints of
$U$ to the endpoints of a gap $V=(\ell,r)$ of $K$, i.e.
\[
\ell=\phi(\ell')\in K,\qquad r=\phi(r')\in K,\qquad (\ell,r)\cap K=\emptyset,
\]
and moreover $\pi_K(\ell)=\pi_K(r)=\psi(s)$. Since $V$ is a gap of $K$, the function $\Pi_K$ is
constant on $[\ell,r]$ with value $\psi(s)$. Because $\tilde\phi([\ell',r'])=[\ell,r]$, we obtain
for every $\tilde{y}\in[\ell',r']$:
\[
\Pi_K(\tilde\phi(\tilde{y}))=\psi(s)=\psi(\Pi_M(\tilde{y})).
\]
This proves \eqref{eq:Pi_identity}. \par
Since $\Pi_K,\Pi_M,\psi$ and $\tilde\phi$ are Lipschitz on intervals, they are differentiable a.e.,
and the chain rule gives for a.e.\ $\tilde{y}\in[\min M,\max M]$:
\[
\Pi_K'(\tilde\phi(\tilde{y}))\,\tilde\phi'(\tilde{y})=\psi'(\Pi_M(\tilde{y}))\,\Pi_M'(\tilde{y}).
\]

By Lemma~\ref{lem:Pi_properties}, there exists a null set $N_K\subseteq[\min K,\max K]$ such that
$\Pi_K'(x)=\1_K(x)$ for all $x\notin N_K$, and similarly $\Pi_M'(x)=\1_M(x)$ for all $x$ outside a null set.
Since $\phi^{-1}:K\to M$ is Lipschitz, it maps null sets to null sets, hence
$\lambda\!\bigl(\phi^{-1}(N_K)\bigr)=0$.
Therefore, for a.e.\ $y\in M$ we have $\tilde\phi(y)=\phi(y)\in K\setminus N_K$ and thus $\Pi_K'(\tilde\phi(y))=1$,
while also $\Pi_M'(y)=1$ and $\Pi_M(y)=\pi_M(y)$.
Consequently,
\[
\tilde\phi'(y)=\psi'(\pi_M(y))\qquad\text{for a.e.\ }y\in M.
\]
Substituting this into \eqref{eq:T_rep_with_psi_prime} yields
\[
(Tf)(y)=\varepsilon\,f(\phi(y))\,\tilde\phi'(y)\qquad\text{for a.e.\ }y\in M,
\]
as desired.
\end{proof}

\section*{Acknowledgments}
This work grew out of fruitful discussions with Professor Th.\ De Pauw. This research was conducted while the author was a Postdoc at the University of Vienna, supported by the Austrian Science Fund (FWF), project 10.55776/F65.

\end{document}